\documentclass[11pt,reqno]{amsart}
\usepackage{amsmath}
\usepackage{amsthm}
\usepackage{amssymb}
\usepackage{amsfonts}
\usepackage{amsxtra}
\usepackage{fullpage}
\usepackage{amscd}
\usepackage{mathrsfs}
\RequirePackage{tikz-cd}
\usepackage{setspace}
\usepackage{tikz}
\usepackage{enumerate}
\usepackage[noautoscale,enableskew]{youngtab}
\usepackage[mathscr]{eucal}
\usepackage{verbatim}

\usepackage{stackrel}

\usepackage{hyperref}
\hypersetup{pdftitle=2lect}

\newtheorem{thm}{Theorem}
\newtheorem{prop}[thm]{Proposition}
\newtheorem{lem}[thm]{Lemma}

\newtheorem{cor}[thm]{Corollary}

\theoremstyle{remark}
\newtheorem{remark}[thm]{Remark}

\newtheorem{example}[thm]{Example}

\theoremstyle{definition}

\newtheorem{defn}[thm]{Definition}
\newtheorem{question}[thm]{Question}

\numberwithin{thm}{section}

\let\nc\newcommand
\let\renc\renewcommand

\newcommand{\Lmod}[1]{#1\text{-}{\mathsf{mod}}}
\newcommand{\LMod}[1]{#1\text{-}{\mathsf{Mod}}}

\DeclareMathOperator{\res}{{\mathbf{res}}}

\DeclareMathOperator{\Ext}{\mathrm{Ext}}

\DeclareMathOperator{\End}{\mathrm{End}}

\DeclareMathOperator{\Rep}{\mathrm{Rep}}

\DeclareMathOperator{\Stab}{\mathrm{Stab}}

\DeclareMathOperator{\QCoh}{\mathrm{QCoh}}
\DeclareMathOperator{\Coh}{\mathrm{Coh}}

\DeclareMathOperator{\ind}{\mathbf{ind}}

\newcommand{\beq}{\begin{equation}\label}
\newcommand{\eeq}{\end{equation}}

\newcommand{\iso}{{\;\stackrel{_\sim}{\to}\;}}

\DeclareMathOperator{\Hom}{\mathrm{Hom}}

\nc{\Z}{\mathbb{Z}}
\newcommand{\N}{\mathbb{N}}
\newcommand{\Q}{\mathbb{Q}}

\newcommand{\C}{\mathbb{C}}

\nc{\rank}{\textrm{rank} \,}
\nc{\ds}{\dots}

\let\mc\mathcal
\let\mf\mathfrak

\nc{\mbf}{\mathbf}
\nc{\LK}{\textsf{Irr}(K)}
\nc{\LW}{\textsf{Irr}(W)}
\nc{\Res}{\mathsf{Res} \, }
\nc{\Ind}{\mathsf{Ind} \, }

\nc{\cont}{\textrm{cont}}

\nc{\eWb}{\mathbf{e}_{W_b}}
\nc{\sing}{\mathrm{sing}}
\nc{\msf}{\mathsf}
\nc{\Ui}{\mc{U}_{i,+}}
\nc{\Uone}{\mc{U}_{1,+}}
\nc{\Utwo}{\mc{U}_{2,+}}

\nc{\minusone}{-1}
\nc{\minustwo}{-2}
\nc{\minusthree}{-3}
\nc{\Mod}{\mathrm{Mod} \,}
\nc{\ms}{\mathscr}
\nc{\Frac}{\mathrm{Frac} \,}
\nc{\ra}{\rightarrow}
\nc{\hra}{\hookrightarrow}
\nc{\lab}{\label}
\nc{\wt}{\widetilde}
\nc{\Tan}{\Theta}
\nc{\ul}{\underline}
\nc{\s}{\mathfrak{S}}
\nc{\g}{\mf{g}}
\nc{\pa}{\partial}
\nc{\tit}{\textit}
\nc{\Maxspec}{\mathrm{Maxspec} \, }
\nc{\gldim}{\mathrm{gl.dim}}
\nc{\rkm}{\mathrm{rk} \, (\mf{m})}
\nc{\sm}{\mathrm{sm}}
\nc{\PD}{\mathbb{PD}}
\nc{\Hilb}{\textrm{Hilb}}
\nc{\id}{\msf{id}}
\nc{\A}{\mathbb{A}}
\nc{\Grel}{\mc{G}^{\mathrm{rel}}}
\nc{\Grat}{\mc{G}^{\mathrm{rat}}}
\nc{\Squo}[1]{\A^{(#1)}}
\nc{\twist}{\mathrm{twist}}
\nc{\Cd}{\mc{C}}
\nc{\Span}{\mathrm{Span}}
\nc{\Grass}{\mathrm{Gr}}
\nc{\Grad}{\mathrm{Gr}^{ad}}
\nc{\V}{\mc{V}}
\nc{\Y}{\mathbb{Y}}
\nc{\rightsim}{\stackrel{\sim}{\longrightarrow}}
\nc{\Psq}{\Ps^2_q}
\nc{\Ps}{\mathbb{P}}
\nc{\prim}{\mathrm{prim}}
\renc{\o}{\otimes}
\renc{\H}{\mathsf{H}}
\nc{\git}{/\!\!/}

\newcommand{\dd}{{\mathscr{D}}}

\newcommand{\Supp}{\mathrm{Supp}}

\newcommand{\Cs}{\C^{\times}}

\newcommand{\bD}{\mathbf{D}}
\newcommand{\QC}{\mathbf{QC}}

\newcommand{\cF}{\mathcal{F}}
\newcommand{\cP}{\mathcal{P}}
\newcommand{\cU}{\mathcal{U}}
\newcommand{\cX}{\mathcal{X}}
\newcommand{\cY}{\mathcal{Y}}
\newcommand{\cZ}{\mathcal{Z}}

\newcommand{\bP}{\mathbf{P}}

\begin{document}


\title{Projective generation for equivariant $\dd$-modules}

\author{Gwyn Bellamy}
\address{School of Mathematics and Statistics, University of Glasgow, University Place, Glasgow, G12 8QQ.}
\email{gwyn.bellamy@glasgow.ac.uk}

\author{Sam Gunningham}
\address{The University of Edinburgh, School of Mathematics,
	James Clerk Maxwell Building,
	Peter Guthrie Tait Road,
	Edinburgh, EH9 3FD, UK}
\email{sam.gunningham@ed.ac.uk}

\author{Sam Raskin}
\address{Department of Mathematics
	The University of Texas at Austin
	2515 Speedway, PMA 8.100
	Austin, TX 78712}
\email{sraskin@math.utexas.edu}

\begin{abstract}
	We investigate compact projective generators in the category of equivariant $D$-modules on a smooth affine variety. For a reductive group $G$ acting on a smooth affine variety $X$, there is a natural countable set of compact projective generators indexed by finite dimensional representations of $G$. We show that 
	only finitely many of these objects are required to generate; thus the category has a single compact projective generator. The proof in the general case goes via an analogous statement about compact generators in the equivariant derived category, which holds in much greater generality and may be of independent interest. We also provide an alternative (more elementary) proof in the case that $G$ is a torus.
\end{abstract}

\maketitle

\section{Introduction}
Let $X$ be a smooth complex affine variety equipped with an action of a complex reductive group $G$. Let $\dd_X$ denote the ring of (algebraic) differential operators on $X$. We consider the category $\QCoh(\dd_X,G)$ of (strongly) $G$-equivariant $\dd_X$-modules; see Section \ref{sec:projeobjectinequivactD} for the definition.

Consider the $\dd_X$-module
\[
\mc P_X = \dd_X /\dd_X \mf g,
\]
where $\mf g \to \dd_X$ is the infinitesimal action map. This object is naturally $G$-equivariant. It represents the functor of invariants on $\QCoh(\dd_X, G)$; known as the functor of quantum Hamiltonian reduction. In particular, it is a compact projective object. One may think of $\mc P_X$ as global differential operators on the quotient stack $[X/G]$. 

Recall that a smooth variety $Y$ is called $D$-affine if the object $\dd_Y$ is a projective generator of the category $\QCoh(\dd_Y)$. Analogously, we say that the $G$-space $X$ (or rather the stack $[X/G]$) is \textit{$D$-affine} if $\mc P_X$ generates the category $\QCoh(\dd_X,G)$, i.e. if every equivariant module has a non-zero invariant element. In this case, we have an equivalence of categories
\[
\Hom(\mc P_X, - ): \QCoh(\dd_X, G) \xrightarrow{\sim} \LMod{R_X}
\]
where $R_X = \End(\mc P_X) \cong (\mc P_X)^G$.

In general, it is a subtle problem to determine whether a given affine $G$-variety is $D$-affine:
	\begin{itemize}
		\item The adjoint action of $GL_n$ on $\mf{gl}_n$ is $D$-affine.
		\vspace{1mm}
		\item The adjoint action of $SL_n$ on $\mf{sl}_n$ is not $D$-affine.
		\vspace{1mm}
		\item The scaling action of $\C^\times$ on $\mathbb{A}^1$ is not $D$-affine.
	\end{itemize}

Thus, instead of asking whether a particular projective module is a compact generator for $\QCoh(\dd_X,G)$, a more fruitful approach to understanding this category would be to ask whether it admits \textit{some} compact projective generator.   

\begin{question}\label{thequestion0}
	 For any given reductive group $G$ and smooth affine $G$-variety $X$, does the category $\QCoh(\dd_X,G)$ admit a compact projective generator? 
\end{question}

A positive answer to this question would imply that there exists a Noetherian $\C$-algebra $R$ such that $\QCoh(\dd_X,G) \simeq \LMod{R}$. This implies for instance that $\QCoh(\dd_X,G)$ admits a finite block decomposition, which is a priori not obvious. 

\subsection{Induced modules}

Let $\Rep(G)$ denote the category of finite dimensional representations of $G$. A countable projective generating set for $\QCoh(\dd_X,G)$ can be constructed by inducing representations of $G$.  More specifically, given $V \in \Rep(G)$ we define
\[
\mc P_X(V) := \dd_X \otimes_{\mc U \mf g} V
\]
These objects represent the functor which assigns the $V$-multiplicity space of an equivariant $\dd_X$-module; as such, they are compact and projective. It is clear from this characterization that the collection of objects $\mc P_X(V)$ generate $\QCoh(\dd_X,G)$, as $V$ ranges over the set of isomorphism classes of irreducible representations. 


Question \ref{thequestion0} is equivalent to:

\begin{question}\label{thequestion}
	For any given reductive group $G$ and smooth affine $G$-variety $X$, does there exist a finite dimensional representation $V \in \Rep(G)$ such that $\mc P_X(V)$ generates $\QCoh(\dd_X,G)$? 
\end{question}

We will show that the answer to Question \ref{thequestion} (and thus also to Question \ref{thequestion}) is yes.

In particular, we have that 
\[
\QCoh(\dd_X,G) = \LMod{R_X(V)}
\]
where $R_X(V) = \End_{\QCoh(\dd_X,G)}(\mc P_X(V))^{\mathrm{op}}$. Concretely, this means that for each smooth affine $G$-variety $X$, there is a finite set $\{V_1, \ldots, V_k\}$ of irreducible representations of $G$ such that each equivariant $\dd$-module $\ms{M}$ contains some $V_i$ appearing with non-zero multiplicity. Here we consider $\ms{M}$ as a $G$-module using the equivariant structure. 


\begin{remark}
	Note that the answer to the analogous question for the category $\QCoh(\mc O_X,G)$ of equivariant quasi-coherent sheaves on $X$ is a clear no. Equivariant structures on skyscraper sheaves at a point $x\in X$ correspond to representations of $H = Stab_G(x)$. Thus, as long as there is a point where the stabilizer $H$ contains a torus, the category $\QCoh(\mc O_X,G)$ cannot admit a compact projective generator.
\end{remark}

\subsection{Main Results}
The main purpose of this paper is to answer Question~\ref{thequestion} (in the affirmative).

\begin{thm}\label{thm:mainprofconj}
	Let $X$ be a smooth affine $G$-variety.	Then there exists $V \in \Rep(G)$ such that $\mc P_X(V)$ generates $\QCoh(\dd_X,G)$. 
\end{thm}

Analogous to the objects $\cP_X(V)$ in $\QCoh(\dd_X,G)$, we have a family of compact objects $\bP_X(V)$ of the triangulated category $\bD(X)^G$; see Section \ref{sec:compactgenerderived} for further details. The proof of Theorem \ref{thm:mainprofconj} is via the following analogous statement in the equivariant derived category:  

\begin{thm}\label{thm:derived}
	If $G$ is a reductive group, and $X$ a smooth affine $G$-variety, there is a finite dimensional representation $V$ such that $\bP_X(V)$ is a compact generator of $\bD(X)^G$.
\end{thm}

In fact, Theorem \ref{thm:derived} is a special case of a more general statement. 

\begin{thm}\label{thm:dg compact generator}
	Let $\cX$ be an Artin stack of finite type with affine diagonal. Then derived category of $D$-modules $\bD(\cX)$ has a compact generator. 
\end{thm}

\begin{remark}
	It is known that that $\bD(\cX)$ is compactly generated (see \cite{drinfeld_finiteness_2013}). The assertion above is that a single compact object suffices to generate.
\end{remark}

We present a proof of Theorem \ref{thm:dg compact generator} in Subsection \ref{sec:proof-of-theorem-refthmdg-compact-generator}. In Section~\ref{sec:compactgenerderived0} we give a self-contained proof of Theorem \ref{thm:derived} which we believe explains why such a generator exists. Namely, the variety $X$ can be stratify by locally closed subvarieties $X_i$, each of which admits a compact generator for the equivariant derived category. This collection is then used to build a compact generator for $X$.

We will also present a different proof in the case when $G$ is a torus in Section \ref{sec:torus-case} which does not require passing to the equivariant derived category.

\subsection{Computations}

Once one knows that the category $\QCoh(\dd_X,G)$ admits a compact generator induced from a representation $V$, it begs the question of which $V$ can be chosen as a generator. A partial answer is given by an easy application of Luna's slice theorem which shows that to find a generating representation $V$, one can reduce to studying the nullcones for each Luna slice.

Luna's theorem associates to the pair $(X,G)$ a finite collection $C_G(X)$ of pairs $(H,N_H)$, where $H$ is a reductive subgroup of $G$ occurring as a stabilizer of a point $x$ (with closed $G$-orbit) in $X$ and $N_H$ is a $H$-stable complement to $\mf{g} \cdot x$ in $T_x X$. Let $N_H'$ denote the $H$-complement to $N_H^H$ in $N_H$. If $\mc{N}(N_H')$ denotes the nilcone in the $H$-representation $N_H'$ then we say that a $H$-representation $W$ is a projective generator relative to $\mc{N}(N_H')$ if $\Hom_{(\dd,H)}(\mc{P}_{N_H'}(W),M) \neq 0$ for all non-zero $H$-equivariant $\dd$-module on $N_H'$ supported on $\mc{N}(N_H')$. 

\begin{thm}\label{thm:mainprojgenslice}
	The representation $V$ is a projective generator of $\Coh(\dd_X,G)$ if and only if $V |_H$ is a projective generator relative to $\mc{N}(N_H')$ for all $(H,N_H) \in C_G(X)$. 
\end{thm}




\subsection{Motivation}
The category $\QCoh(\dd_X,G)$ should be thought of as a quantization of the cotangent stack
\[
T^\ast([X/G]) = [\mu_X^{-1}(0) /G]
\]
where $\mu_X: T^\ast X \to \mf g^\ast$ is the moment map. Such examples arise naturally in geometric and symplectic representation theory.
Here are some motivating examples to keep in mind:
\begin{itemize}
	\item If $X=\mf g$ (or $G$) with the adjoint (conjugation) action, the category $\QCoh(\dd_X,G)$ is the home of Lusztig's character sheaves, studied in their $\dd$-module incarnation by Ginzburg. Of particular interest are the cuspidal $\dd$-modules, which correspond to the certain IC sheaves on distinguished nilpotent orbits. 
	\item In the case $G=GL_n$, one can modify the previous example by taking $X=\mf {gl}_n \times \C^n$. Objects of $\QCoh(\dd_X,G)$ are known as \emph{mirabolic} $\dd$-modules. A certain localization of this category is closely related to representations of the rational Cherednik algebra and the geometry of the Hilbert scheme of points in the plane.
	\item Given a quiver $Q = (Q_1, Q_0)$ and fixing a dimension vector $\alpha \in \N^{Q_0}$ gives rise to another example of a vector space 
	\[
	X = \bigoplus_{f\in Q_1} \Hom(\C^{\alpha_{s(f)}},\C^{\alpha_{t(f)}})
	\]
	equipped with the action of $G=\prod_{i\in Q_0} GL_{\alpha_i}$. Localizations of the category $\QCoh(\dd_X,G)$ give quantizations of Nakajima quiver varieties.
	\item In the case where $G=T$ acts linearly on a vector space $X$, localizations of the category $\QCoh(\dd_X,G)$ give quantizations of hypertoric varieties.
\end{itemize}
	
\subsection{Example: the adjoint representation and generalized Springer theory}
The category $\QCoh(\dd_{\mf g},G)$ for $G$ a connected semisimple group has been studied in \cite{GunnighamAbelian}. It was shown that there is a finite orthogonal decomposition of $\QCoh(\dd_{\mf g},G)$ indexed by the \emph{cuspidal data} associated to $G$ (in the sense of Lusztig's generalized Springer correspondence). Each block corresponding to a given cuspidal datum is of the form $\dd_{\mf z} \rtimes \Gamma$ where $\Gamma$ is a finite group acting by reflections on a vector space $\mf z$. In particular, there is a given set of projective generators of $\QCoh(\dd_{\mf g},G)$ indexed by cuspidal data.

On the other hand, there are finitely many irreducible representations $\{V_1, \ldots , V_k\}$ such that $\mc P_{\mf g}(V_1), \ldots, \mc P_{\mf g}(V_k)$ generate $\QCoh(\dd_{\mf g},G)$. It is not clear at all what the relationship is between these two sets of generators. The following examples give a sense of the nature of this case (see \cite{gunningham_generalized_} for further details):
\begin{enumerate}
	\item If $G=PGL_n$, then $\mc P_{\mf g}(\C)$ already generates the category (i.e. it is enough to take the trivial representation).
	\item If $G=SL_n$, then the fundamental representations (together with the trivial representations) form a minimal set of irreducible representations such that the corresponding $\dd$-modules generate.
	\item In general, $\mc P_{\mf g}(\C)$ generates the \emph{Springer block} of the category with respect to the generalized Springer decomposition; this is a consequence of a theorem of Hotta and Kashiwara \cite{HottaKashiwara}. Thus the $G$-variety $\mf g$ is $D$-affine only if there are no non-trivial cuspidal data; this happens only when $G$ is of type $A$ with connected center.
\end{enumerate}

The building blocks of generalized Springer theory are called cuspidal objects. The cuspidal $D$-modules in $\QCoh(\dd_{\mf g},G)$ are very special: they are necessarily supported on the nilpotent cone, as are their Fourier transform (in fact, this characterizes cuspidality). 

Another notable feature of cuspidal $\dd$-modules is that they have finite multiplicity. That is, each irreducible $G$-representation appears with finite multiplicity in the module. Thus one can associate to any such cuspidal $\dd$-module its \emph{character}: a sequence of numbers indexed by irreducible representations, recording the dimension of each multiplicity space. It is interesting to note that, though the cuspidal $\dd$-modules are regular holonomic and thus have a purely topological incarnation as perverse sheaves, this character does not appear to be accessible from a purely topological perspective. 

More generally, if $\ms{M}$ is any irreducible $G$-equivariant $\dd$-module supported on the nilpotent cone then it is also (weakly) equivariant for the scaling $\Cs$-action on $\mf{g}$. The $G \times \Cs$-multiplicities spaces of $\ms{M}$ will be finite dimensional and thus $\ms{M}$ has a well-defined $G \times \Cs$-character. This character has been computed for $G = SL_n$ in \cite{MinimalSupport}; see also \cite{SomeComb}. Characters can similarly be defined for equivariant $\dd$-modules on certain classes of $G$-representations, as done in \cite{RaicuSymmetric,RaicuVeronses,binarycubicDmod}. 

\subsection{Further questions}


It is conjectured that if $X$ is a smooth projective variety over $\C$ which is $\dd$-affine, then $X \cong G / P$ for some reductive group $G$ and parabolic subgroup $P$. Based on the results of this article, it is natural to ask:
\begin{enumerate}
	\item[-] For which smooth quasi-projective varieties over $\C$ does the category $\Coh(\dd_X,G)$ admit a projective generator?  
\end{enumerate}

In general, the modules $\mc P_X(V)$ appear hard to compute. For a given affine $G$-variety $X$ and representation $V$, it would be desirable to have a criterion to check whether $\mc P_X(V)$
\begin{enumerate}
	\item[-] is non-zero;
	\item[-] is indecomposable;
	\item[-] is holonomic, or has a holonomic direct summand or submodule.
\end{enumerate}

\subsection{Acknowledgements}

We would like to thank Travis Schedler for useful comments on an earlier draft of this paper. We would also like to thank Toby Stafford for many fruitful discussions.  The first author would like to thank Kevin Coulembier for a useful discussion. 

\section{Projective generators in the abelian category}\label{sec:projeobjectinequivactD}

Let $G$ be a reductive algebraic group over $\C$ and $X$ a smooth affine variety over $\C$ with a regular action of $G$ i.e. a \textit{$G$-variety}. Let $\sigma \colon G \times X \rightarrow X$ be the action map, $p_2 \colon G \times X \rightarrow X$ projection onto the second factor and $s \colon X \hookrightarrow G \times X$ the embedding $s(x) = (e,x)$. 

\begin{defn}\label{defn:equivDmod}
	A $G$-equivariant $\dd$-module on $X$ is a pair $(\ms{M},\theta)$, where $\ms{M}$ is a quasi-coherent $\dd_X$-module and $\varphi \colon p_2^* \ms{M} \stackrel{\sim}{\longrightarrow} \sigma^* \ms{M}$ is an isomorphism of $\dd_{G\times X}$-module satisfying the cocycle condition of \cite[Definition 11.5.2]{HTT} and the rigidity condition $s^* \varphi = \mathrm{Id}_{\ms{M}}$. 
\end{defn}

The category of all $G$-equivariant $\dd$-modules on $X$ is denoted $\QCoh(\dd_X,G)$ and morphisms in $\QCoh(\dd_X,G)$ will be denoted $\Hom_{(\dd_X,G)}(L,M)$.   

\begin{defn}\label{defn:projVDmod}
	For each finite dimensional $G$-module $V$, we define 
	$$
	\mc{P}_X (V) = \dd_X \o V / \dd_X \{ \nu(x) \o v - 1 \o \Phi(x)(v) \ | \ v \in V \},
	$$
	where $\nu \colon \mf{g} \rightarrow \dd_X$ is the comoment map and $\Phi : \mf{g} \rightarrow \End_{\C}(V)$ is the action.
\end{defn}

\begin{lem}
	Assume that $X$ is a smooth affine $G$-variety.
	\begin{enumerate}
		\item[(i)] $\mc{P}_X(V)$ has a canonical $G$-equivariant structure and is projective in $\QCoh(\dd_X,G)$.
		\item[(ii)] The category $\QCoh(\dd_X,G)$ has enough projective. 
	\end{enumerate}
\end{lem}

\begin{proof}
	Both statements follow from the fact that 
	$$
	\Hom_{(\dd_X,G)}(\mc{P}_X(V), L) = \Hom_G(V,\Gamma(X,L))
	$$
	for any $L$ in $\QCoh(\dd_X,G)$.
\end{proof}

Recall that a subset $Z$ of $X$ is \textit{$G$-saturated} if it is the preimage of a subset of $X \git G$. If $Z$ is a closed $G$-saturated subset of $X$ then we say that $V$ is a projective generator relative to $Z$ if 
$$
\Hom_{(\dd_X,G)}(\mc{P}_X(V), L) = \Hom_G(V,\Gamma(X,L)) \neq 0
$$
for all $L \in \QCoh(\dd_X,G)$ supported on $Z$. It suffices to consider $L \in \Coh(\dd_X,G)$; in fact we may assume that $L$ is irreducible. 

\section{Compact generators in the derived category}\label{sec:compactgenerderived0} 

\subsection{Background: the equivariant derived category}
Given an affine algebraic group $G$ acting on a smooth quasi-projective algebraic variety $X$, we denote by $\bD(X)^G$ the (unbounded) equivariant derived category of $\dd$-modules (or equivalently, the derived category of $\dd$-modules on the quotient stack $[X/G]$).  This is a triangulated category with a $t$-structure whose heart is $\QCoh(\dd_X,G)$. 

There are various approaches to defining this category in the literature, see \cite{bernstein_equivariant_1994},\cite{bernstein_localization_1995}, \cite{beilinson_quantization_????}, Chapter 7, or \cite{gaitsgory_crystals_2011}. For the purposes of this paper, we will only need some formal properties, namely:
\begin{itemize}
	\item Given a $G$-equivariant morphism $f:X\to Y$, we have a pair of functors 
	\[
	f_\ast:\bD(X)^G \leftrightarrows \bD(Y)^G: f^!
	\]
	If $f$ is an open embedding then $f^! = f^\ast$ is left adjoint to $f_\ast$ and if $f$ is a closed embedding then $f^!$ is right adjoint to $f_\ast = f_!$.
	\item If $i:Z \hookrightarrow X$ is a $G$-equivariant closed embedding and $j:U \to X$ the complementary open embedding, for each object $\ms{M} \in \bD(X)^G$, there is an exact triangle:
	\[
	i_\ast i^! \ms{M} \to \ms{M} \to j_\ast j^! \ms{M}
	\] 
	\item There are adjoint induction and restriction functors:\footnote{In fact, there are two such pairs of functors, corresponding to the realization of $\bD(X)^G$ as either left or right $\dd$-modules.} 
	\[
	\ind: \QC(X)^G \leftrightarrows \bD(X)^G: \res
	\]
	($\ind$ is left adjoint). Here $\QC(X)^G$ denotes the derived category of $G$-equivariant quasi-coherent sheaves on $X$. 
\end{itemize}

\subsection{Compact generators}\label{sec:compactgenerderived}

Recall that an object $d$ of a triangulated category $\bD$ is called compact if $\Hom_{\bD}(d,-)$ preserves direct sums. A collection of objects $d_i, i \in I$ is said to generate $\bD$ if for any non-zero object $e$ of $\bD$, we have
\[
\Hom_\bD(d_i,e[k]) \neq 0 
\]
for some $i\in I$ and some integer $k\in \mathbb Z$ (denoting cohomological shift). 

Now suppose $G$ is a reductive group and $X$ a smooth affine $G$-variety. Given a representation $V\in \Rep(G)$, there is a derived analogue of the modules $\cP_X(V)$ which we denote $\bP_X(V)$, defined by
\[
\bP_X(V) = \ind(V \otimes \mathcal O_X)
\]
As in the abelian category setting, if $X$ is affine and $G$ reductive, the objects $\bP_X(V)$ form a set of compact generators of $\bD(X)^G$ as $V$ ranges over finite dimensional representations of $G$. In fact, this holds even under the assumption that $X$ is quasi-affine (and $G$ need not be reductive) as it is still the case that the objects $\mathcal O_X \otimes V$ form a set of compact generators for $\QC(X)^G$ in that case. 

The following observation is immediate from the defining adjunction properties of $\bP_X(V)$ and $\cP_X(V)$ together with the fact that the object $V\in \Rep(G)$ is projective.
\begin{lem}\label{lem:derived vs abelian}
	For any $V\in \Rep(G)$ and $\ms{M} \in \QCoh(\dd_X,G)$, we have:
	\[
	\Hom_{\bD(X)^G}(\bP_X(V),\ms{M}[k]) \cong \begin{cases} \Hom_{\Rep(G)}(V,\ms{M}), \quad \text{if $k = 0$}\\ 0, \quad \text{otherwise.}\end{cases}
	\]
\end{lem}

Thus we obtain the following result relating generators in the abelian and derived categories:
\begin{prop}
	Suppose $V \in \Rep(G)$ is a finite dimensional representation with the property that $\bP_X(V)$ is a (compact) generator of $\bD(X)^G$. Then $\cP_X(V)$ is a (compact, projective) generator of $\QCoh(\dd_X,G)$.
\end{prop}

It follows that, to prove Theorem \ref{thm:mainprofconj}, it suffices to prove the derived analogue, Theorem \ref{thm:derived}

\subsection{Proof of Theorem \ref{thm:derived}}
The basic idea behind the proof is to stratify the variety $X$ in to locally closed subvarieties $X_i$, each of which admits a compact generator for the equivariant derived category, then use this collection to build a compact generator for $X$.

\begin{prop}\label{prop:stratification}
	Let $X$ be a smooth affine $G$-variety. There is a sequence of $G$-stable closed subvarieties:
	\[
	Y_0 \subseteq \ldots \subseteq Y_n = X
	\]
	where the successive complements $X_i = Y_i - Y_{i-1}$ are quasi-affine, smooth subvarieties with the following property:
	\begin{itemize}
		\item there is a finite dimensional representation $W_i \in \Rep(G)$ such that $\bP_{X_i}(W_i) \in \bD(X_i)^G$ is a compact generator.
	\end{itemize}
\end{prop}
Proposition \ref{prop:stratification} may be proved in a similar way to Theorem \ref{thm:mainprojgenslice} using Luna slices \cite{Luna} and Hesselink's stratification of the nullcone \cite[Theorem 4.7]{HesselinkDesing}. These results combined imply that one can choose a stratification of $X$ as above such that there exists an \'etale equivariant surjection $G \times_{P_i} U_i \to X_i$, where $P_i$ is a parabolic of a reductive subgroup $H_i$ of $G$ and $U_i$ is a smooth quasi-affine variety acted upon by the reductive quotient $P_i / U(P_i)$. It is also a special case of Proposition \ref{prop:strat of stacks} below, using Lemma 10.3.9 in \cite{drinfeld_finiteness_2013}.

Let $\alpha_i \colon Y_i \hookrightarrow X$ and $\beta_i \colon X_i \hookrightarrow X$ be the (locally) closed embeddings. We can now find a suitable representation $V$ as follows. Let $Q_i \in \bD(X)^G$ be a compact object extending $\bP_{X_i}(W_i)$, i.e. $\beta_i^!(Q_i) = \bP_{X_i}(W_i)$ (for example, one can take the induced $D$-module associated to the coherent sheaf $W_i \otimes \mathcal O_{Y_i}$). As with any finite collection of compact objects, we can find a single finite dimensional representation $V$ such that $Q_i \in \langle \bP_X(V) \rangle$ for each $i=0, \ldots , n$. In other words, each $Q_i$ can be expressed as a direct summand of a complex of copies of $\bP_X(V)$.

Now suppose that $\ms{M} \in \bD(X)^G$ is such that $R\Hom_{\bD(X)^G}(\bP_X(V),\ms{M}) \simeq 0$. By construction, we must have $R\Hom_{\bD(X)}(Q_i,\ms{M}) \simeq 0$ for all $i= 0, \ldots , n$. We will show by induction that $\alpha_{i\ast}\alpha_i^!(\ms{M}) \simeq 0$ for each $i= 0, \ldots , n$. 

To begin with, note that $\alpha_0 = \beta_0$ is a closed inclusion of a smooth subvariety $Y_0$, and $Q_0 = (\alpha_0)_\ast \bP_{Y_0}(W_0)$. There is an exact triangle:
\[
(\alpha_0)_\ast (\alpha_0)^! \ms{M} \to \ms{M} \to (\mathring{\alpha}_0)_\ast(\mathring{\alpha}_0)^! \ms{M}
\]
Applying the functor $R\Hom(Q_0, -)$ to this triangle, and noting that the middle term (by assumption) and the right hand term (as $Q_0$ is supported on $Y_0$) must vanish, we see that 
\[
0 \simeq R\Hom(Q_0,(\alpha_0)_\ast (\alpha_0)^! \ms{M}) \simeq R\Hom(\bP_{X_0}(W_0), (\alpha_0)^! \ms{M}),
\]
and thus $(\alpha_0)^!\ms{M} \simeq 0$ as desired. 

Now suppose we want to show that $(\alpha_i)^!\ms{M} \simeq 0$ for $i>0$. First we have the exact triangle
\[
(\alpha_{i})_\ast(\alpha_i)^! \ms{M} \to \ms{M} \to (\mathring{\alpha}_i)_\ast(\mathring{\alpha}_i)^! \ms{M}
\]
Applying $R\Hom_{\bD(X)}(Q_i,-)$ as before, we see that the middle and right hand terms vanish. Thus
\[
R\Hom_{\bD(X)}(Q_i,(\alpha_{i})_\ast(\alpha_i)^! \ms{M}) \simeq 0
\]
On the other hand, there is another exact triangle:
\[
(\alpha_{i-1})_\ast (\alpha_{i-1})^! \ms{M} \to (\alpha_i)_\ast(\alpha_i^!)\ms{M} \to (\beta_i)_\ast (\beta_i)^! \ms{M}
\]
By the inductive hypothesis, the left hand term vanishes, and thus the right hand arrow is an isomorphism. Thus we have that
\begin{align*}
0 &\simeq  R\Hom_{\bD(X)^G}(Q_i,(\alpha_i)_\ast (\alpha_i)^! \ms{M}) \\
&\simeq R\Hom_{\bD(X)^G}(Q_i,(\beta_i)_\ast (\beta_i)^! \ms{M})\\
&\simeq R\Hom_{\bD(X)^G}(\bP_{X_i}(W_i), (\beta_i)^! \ms{M})
\end{align*}
where the last isomorphism uses that everything is supported on $Y_i$ and thus $\beta_i$ behaves as an open embedding. By the assumption on $W_i$, we see that $\beta_i^!\ms{M} \simeq 0$, and thus $\alpha_i^! \ms{M} \simeq 0$ as required.

\section{Reduction to the nullcone}

In this section we explain how one can reduce Question~\ref{thequestion} to a question about projective modules that are generators relative to $\dd$-modules supported on the nullcone of a $G$-representation. 

\subsection{Luna's stratification}\label{sec:LunastrataProjgen}

If $x \in X$ such that the orbit $G \cdot x$ is closed, then the stabilizer $H = \Stab_G(x)$ is reductive by Matsushima's Theorem. Let $N_H$ be a $H$-stable complement to $\mf{g} \cdot x$ in $T_x X$. Following Luna, we call the pair $(H,N)$ a \textit{model} of $X$ at $x$. Two models $(H,N_H)$ and $(K,N_K)$ are equivalent if there exists $g \in G$ such that $g H g^{-1} = K$ and ${}^g N_H \simeq N_K$ as $K$-modules. The set of all equivalence classes of models is denoted $C_G(X)$. For $(H,N_H) \in C_G(X)$, let $(X \git G)_{(H,N_H)}$ denote the set of all points $y$ such that if $x \in \pi^{-1}(y)$ belongs to the unique closed $G$-orbit of the fibre then the model of $X$ at $x$ is in the equivalence class of $(H,N_H)$. By Corollary 3 and Corollary 4 of \cite{Luna}, 
$$
X \git G = \bigsqcup_{(H,N_H) \in C_G(X)} (X \git G)_{(H,N_H)}
$$
is a finite stratification, with smooth locally closed strata $(X \git G)_{(H,N_H)}$. Recall that Let $N_H'$ denote the $H$-complement to $N_H^H$ in $N_H$.

If $X$ is a $G$-representation then it has been shown in \cite[Lemma 5.5]{SchwarzLiftingHomo} that each stratum is connected and $N_H$ is, up to isomorphism, uniquely determined by $X$ and $H$. Thus, $C_G(X)$ is identified with a subset of the set of conjugacy classes of reductive subgroups of $G$.  

When $X$ is a $G$-representation, the \textit{nullcone} is $\mc{N}(X) := \pi^{-1}(0)$. 

\subsection{The proof of Theorem \ref{thm:mainprojgenslice}}\label{sec:proofofmainprojgenslice}

The goal of this section is to prove Theorem~\ref{thm:mainprojgenslice}. We will actually show a more precise statement holds in the proof. Let $\varphi \colon X \rightarrow Y$ be an equivariant morphism between affine $G$-varieties. Recall that $\varphi$ is \textit{excellent} if 
\begin{enumerate}
	\item[(a)] $\varphi$ is \'etale. 
	\item[(b)] The induced morphism $\varphi\git G \colon X \git G \rightarrow Y \git G$ is \'etale. 
	\item[(c)] The induced map $X \rightarrow Y \times_{Y \git G} X \git G$ is a $G$-equivariant isomorphism.
\end{enumerate}

\begin{lem}\label{lem:projgenexcellentmap}
Assume that $\varphi \colon X \rightarrow Y$ is an excellent morphism between smooth affine $G$-varieties, $Z \subset Y$ a $G$-saturated closed subset and $Z' = \varphi^{-1}(Z)$. Then the $G$-module $V$ is a projective generator relative to $Z$ if and only if it is a projective generator relative to $Z'$. 
\end{lem}

\begin{proof}
Let $L \in \Coh(\dd_Y,G)$. First, we note that 
$$
\Hom_{(\dd_Y,G)}(\mc{P}_Y(V),L) \neq 0 \quad \Leftrightarrow \quad \Hom_{(\dd_X,G)}(\mc{P}_X(V),\varphi^* L) \neq 0
$$
because $\Gamma(X,\varphi^* L) = \C[X]^G \o_{\C[Y]^G} \Gamma(Y,L)$ and $\C[X]^G$ is faithfully flat over $\C[Y]^G$. 

Assume now that $V$ is a projective generator relative to $Z'$. If $L$ is supported on $Z$ then $\Hom_{(\dd_X,G)}(\mc{P}_X(V),\varphi^* L) \neq 0$ since $\varphi^* L$ is supported on $Z'$. Thus, $V$ is a projective generator relative to $Z$. 

Conversely, assume that $V$ is a projective generator relative to $Z$. Let $M$ be a coherent $G$-equivariant $\dd_X$-module supported on $Z'$. Then $\varphi_! M = \varphi_* M$ is a coherent $G$-equivariant $\dd_Y$-module supported on $Z$ and $\Gamma(Y,\varphi_! L) = \Gamma(X,L)$. Thus, $\Hom_G(V,\Gamma(Y,\varphi_! L)) \neq 0$ implies that $\Hom_G(V,\Gamma(X,L)) \neq 0$ as required. 
\end{proof}

Let $H \subset G$ be a reductive subgroup and $Y$ a smooth affine $H$-variety. Set $X = G \times_H Y$ and write $i \colon Y \hookrightarrow X$ for the closed embedding $i(y) = (1,y)$. As shown in \cite[Proposition 9.1.1]{mirabolicHam}, pullback defines an equivalence $i^* \colon \Coh(\dd_X,G) \rightarrow \Coh(\dd_Y,H)$. The following result is established in the proof of \cite[Lemma 3.6]{BellBoosSemisimple}. 

\begin{lem}\label{lem:istarVrestrict}
The equivalence $i^*$ satisfies $\Hom_G(V,L) = \Hom_H(V |_H, i^* L)$.
\end{lem}

Consider now the case where $X$ is a $G$-module. Then $X = X^G \oplus X'$ for some (unique) $G$-complement $X'$, and $\mc{N}(X) = \mc{N}(X') \times \{ 0 \}$. Moreover, we are interested in this situation because the subvariety $\mc{N}(X') \times X^G$ will correspond to the lowest stratum of Luna's stratification for $X$.  

\begin{lem}\label{lem:projgenreomveinvariants}
The representation $V$ is a projective generator relative to $\mc{N}(X') \times X^G \subset X$ if and only if it is a projective generator relative to $\mc{N}(X') \subset X'$. 
\end{lem}

\begin{proof}
Assume that $V$ is a projective generator relative to $\mc{N}(X') \times X^G$. Let $L$ be a coherent $G$-equivariant $\dd_{X'}$-module supported on $\mc{N}(X')$. Then $L' = L \boxtimes \dd_{X^G}$ is a $G$-equivariant coherent $\dd_X$-module supported on $\mc{N}(X') \times X^G$. Hence $\Hom_G(V,L \boxtimes \dd_{X^G}) \neq 0$. But 
$$
\Hom_G(V,L \boxtimes \dd_{X^G}) = \Hom_G(V,L)\otimes \dd_{X^G}
$$
and hence $\Hom_G(V,L) \neq 0$. 

Conversely, assume that $V$ is a projective generator relative to $\mc{N}(X') \subset X'$. Let $M$ be a coherent $G$-equivariant $\dd_X$-module supported on $\mc{N}(X') \times X^G$. Let $U \subset M$ be a non-zero finite dimensional $G$-submodule. Then $\dd_{X'} \cdot U =: L'$ is a coherent $G$-equivariant $\dd_{X'}$-submodule of $M$; it is supported on $\mc{N}(X')$. Hence $\Hom_G(V,L') \neq 0$. But $\Hom_G(V,L') \subset \Hom_G(V,M)$. Thus, $V$ is a projective generator relative to $\mc{N}(X') \times X^G$.   
\end{proof}

We now give the proof of Theorem \ref{thm:mainprojgenslice}.

\begin{proof}
	Since the proof is rather long, we give a short summary. Each irreducible $G$-equivariant $D$-module is supported on the closure of some Luna stratum $X_{(H,N_H)}$ of $X$. The stratum gives rise to a diagram 
	\begin{equation}\label{eq:explainLunaDmodgen}
	\begin{tikzcd}
	& S \ar[r,hook,"i"] \ar[d,"\phi"'] & G \times_H S \ar[d,"\varphi"] \\
	N_H' \ar[r,hook] & N_H & X. 
	\end{tikzcd}
	\end{equation}
	Pulling back along $\varphi$ and $i$, and pushing forward along $\phi$, we deduce that $V$ is a generator relative to modules supported on $X_{(H,N_H)}$ if and only if $V |_H$ is a generator relative to modules supported on the nullcone of $N_H'$. 
	
	To begin the proof, recall from section \ref{sec:LunastrataProjgen} that Luna's slice theorem implies that $X\git G$ has a finite stratification by smooth (possibly disconnected) locally closed strata $(X \git G)_{(H,N_H)}$. Let $X_{(H,N_H)}$ denote the preimage of $(X \git G)_{(H,N_H)}$ in $X$. 
	 
Assume that $L$ is an irreducible $G$-equivariant $\dd_X$-module. Then the support of $L$ is $G$-irreducible i.e. there exists an irreducible component $Z \subset \Supp \, L$ such that $\Supp \, L = G \cdot Z$; if $G$ is connected then we can just take $Z = \Supp \, L$. This implies that the image of $Z$ in $X \git G$ is closed and irreducible. Hence there exists a unique $(H,N_H) \in C_G(X)$ such that $X_{(H,N_H)} \cap \Supp \, L$ is open dense in $\Supp \, L$. Let us say that $L$ is an irreducible $\dd$-module associated to $(H,N_H)$. We will prove that $\Hom_G(V,L) \neq 0$ for all irreducible $L$ associated to $(H,N_H)$ if and only if $V$ is a projective generator relative to $\mc{N}(N_{H}') \subset N_{H}'$. 


First, by Lemma \ref{lem:projgenreomveinvariants}, it suffices to show that $\Hom_G(V,L) \neq 0$ for all irreducible $L$ to $(H,N_H)$ if and only if $V$ is a projective generator relative to $\mc{N}(N_{H}) \times N_H^H \subset N_{H}$. Since $\Supp \, L $ is closed and $G$-stable, we may choose $x \in \Supp \, L$ such that $G \cdot x$ is closed and the stabilizer of $x$ equals $H$. Then Luna's Slice Theorem \cite{Luna} says that there exists a slice $x \in S \subset X$ to the $G$-orbit of $x$ such that the natural map $\varphi \colon G \times_H S \rightarrow X$ has image $U$, an affine $G$-saturated open subset of $X$ and the map $\varphi \colon G \times_H S \rightarrow U$ is excellent. Since $U$ is $G$-saturated, $\Gamma(U,L|_U) = \C[U]^G \otimes_{\C[X]^G} \Gamma(X,L)$. Moreover, since $L$ is irreducible, it has no non-zero sections supported on the complement to $U$. Thus, $\Hom_G(V,\Gamma(X,L)) \neq 0$ if and only if $\Hom_G(V,\Gamma(U,L)) \neq 0$. Therefore, we deduce from Lemma \ref{lem:projgenexcellentmap} that $\Hom_G(V,\Gamma(X,L)) \neq 0$ for all $L$ associated to $(H,N_H)$ if and only if $V$ is a projective generator relative to $(G \times_H S)_{(H,N_H)} \subset G \times_H S$. 

The closed subset $(G \times_H S)_{(N_H)}$ equals $G \times_H S_{(H,N_H)}$. Under the equivalence 
$$
i^* \colon \Coh(\dd_{G \times_H S},G) \iso \Coh(\dd_S,H),
$$
the module $M$ is supported on $(G \times_H S)_{(H,N_H)}$ if and only if $i^* M$ is supported on $S_{(H,N_H)}$. Therefore, we deduce from Lemma \ref{lem:istarVrestrict} that $V$ is a projective generator relative to $(G \times_H S)_{(H,N_H)} \subset G \times_H S$ if and only if $V |_H$ is a projective generator relative to $S_{(H,N_H)} \subset S$. 

Next, Luna's Slice Theorem \cite{Luna} also says that there is a $H$-equivariant map $\phi \colon S \rightarrow N_H$, whose image is a $H$-saturated affine open subset $U'$ of $0$ such that the map $\phi \colon S \rightarrow U'$ is excellent. We deduce once again from Lemma \ref{lem:projgenexcellentmap} that $V |_H$ is a projective generator relative to $S_{(H,N_H)} \subset S$ if and only if $V |_H$ is a projective generator relative to $U_{(H,N_H)}' \subset U'$. Again, since $U' \subset N_H$ is saturated, if $M$ is an irreducible $H$-equivariant $\dd$-module supported on $(N_H)_{(H,N_H)}$ then $\Gamma(U',M) = \C[U']^G \otimes_{\C[N_H]^H} \Gamma(N_H,M)$. Moreover, since $M$ is irreducible, it has no non-zero sections supported on the complement to $U'$. Thus, $\Hom_H(V|_H,\Gamma(N_H,M)) \neq 0$ if and only if $\Hom_H(V|_H,\Gamma(U',M)) \neq 0$.
\end{proof}

\section{The visible case}\label{sec:the-visible-case}

In this section we explain how Theorem~\ref{thm:mainprojgenslice} can be used to give an alternative proof of Theorem~\ref{thm:mainprofconj} that avoids the use of the equivariant derived category in a large class of example - those varieties that are visible. Based on the geometry of $G$-representations, we say that an affine $G$-variety $X$ is \textit{visible} if every fiber $\pi^{-1}(x)$ of the quotient map $\pi: X \rightarrow X \git G$ consists of finitely many $G$-orbits.

\begin{lem}\label{lem:visibleXsmooth}
	Let $X$ be a smooth affine $G$-variety. 
	\begin{enumerate}
		\item[(i)] If $X$ is a $G$-representation then it is visible if and only if $\mc{N}(X)$ consists of finitely many $G$-orbits. 
		\item[(ii)] If $X$ is visible then each slice $N_H$ is a visible $H$-representation. 
	\end{enumerate}
\end{lem}

\begin{proof}
	Part (i) is Corollary 3 to Proposition 5.1 in \cite{PopovVinberg}. 
	
	Part (ii): by part (i) it suffices to show that $\mc{N}(N_H)$ consists of finitely many $H$-orbits. Let $x \in X_{(H,N_H)}$ with closed $G$-orbit and stabiliser $H$. Write $\pi_H \colon N_H \rightarrow N_H \git H$ for the quotient map. Then remark 2 of \cite[Section III.1]{Luna} says that 
	$$
	\pi^{-1}(\pi(x)) \simeq G \times_H \pi^{-1}_H(0) = G \times_H \mc{N}(N_H)
	$$
	as $G$-spaces. If $\pi^{-1}(\pi(x))$ consists of finitely many $G$-orbits, it follows that $\mc{N}(N_H)$ consists of finitely many $H$-orbits.  
\end{proof}

\begin{example}
	\begin{enumerate}
		\item A $G$-space with finitely many orbits is visible, and the result is easy to see in this case since $\QCoh(\dd_X,G)$ has only finitely many irreducible objects.
		\item The irreducible visible representations have been classified by Kac \cite{KacNilpotentorbits}. 
		\item Other examples of visible $G$-spaces include the adjoint representations $\mf g$ and Vinberg's $\theta$-representations; in particular the representation spaces of a cyclic quiver.   
	\end{enumerate}
\end{example}

In the visible case, Theorem \ref{thm:mainprojgenslice} implies:

\begin{cor}
	If $X$ is visible then there exists a projective generator $V$ of $\Coh(\dd_X,G)$. 
\end{cor}

\begin{proof}
	By Theorem \ref{thm:mainprojgenslice} it suffices to show that there exists $V$ such that $V |_H$ is a projective generator relative to $\mc{N}(N_H')$ for all $(H) \in C_H(X)$. By Lemma \ref{lem:visibleXsmooth}, the fact that $X$ is visible implies that each nullcone $\mc{N}(N_H')$ consists of finitely many $H$-orbits. Therefore, it suffices to show that if $X$ is a visible $G$-module then there exists a projective generator $V$ relative to $\mc{N}(X)$. Since there are only finitely many $G$-orbits in $\mc{N}(X)$, the category of coherent $G$-equivariant $\dd_X$-modules, supported on $\mc{N}(X)$, consists of holonomic $\dd$-modules, and there are only finitely many irreducible modules. Call these $L_1, \ds, L_k$. Then we can easily choose $V$ such that $\Hom_G(V,L_i) \neq 0$ for each $i$. 
\end{proof}

\begin{remark}
	The existence of a projective generator was also shown in \cite[Theorem 1.5]{BellBoosSemisimple} under the restriction that $X$ is a visible $G$-representation and that the map $\mu^{-1}(0) \git G \rightarrow X \git G \times X^* \git G$ is finite map. However, examples show that this is a very strong restriction.  
\end{remark}

\section{Torus Case}\label{sec:torus-case}

In this section we give an alternative proof Theorem \ref{thm:mainprofconj} in the case where $G$ is a torus. Let $T =(\C^\times)^r$ be a torus of rank $r$ acting regularly on a smooth affine variety $X$. By Theorem \ref{thm:mainprojgenslice}, we may assume that $X \simeq \C^n$ with $T$ acting linearly. Let $\lambda_1,\ldots, \lambda_n \in X_\bullet(T) = \Z^r$ be the weights of this action. The $\Z$-submodule of $\Z^r$ generated by $\lambda_1,\ldots, \lambda_n$ is denoted $A$. We write $x_1, \ldots, x_n$ for the coordinate functions on $X$. We begin with:

\begin{lem}\label{lem:everywheresupportedTmod}
	Let $\mc M$ be a $T$-equivariant $\dd_X$-module, and suppose that for all weight vectors $m\in \mc M$, $x_km \neq 0$ and $\partial_k m \neq 0$. If $m \in \mc M$ is any non-zero weight vector of weight $\lambda$ then there exists a non-zero weight vector $m' \in \mc M$ of weight $\lambda'$ for any $\lambda' \in \lambda + A$.
\end{lem}

\begin{proof}
	Given 
	\[
	\lambda' = \lambda + N_1 \lambda_1 + \cdots + N_n \lambda_n \in \lambda + A,
	\] 
	let $R_k$ denote either $x_k^{N_k} \in \dd_X$ if $N_k \geq 0$, or $\partial_k^{-N_k} \in \dd_X$ if $N_k <0$. Then consider the vector
	\[
	m' = R_1 R_2 \cdots R_n m \in \mc M
	\]
	We will show that $m'$ is the required weight vector of weight $\lambda'$. First note that operating by $x_k$ takes the $\mu$-weight space in to the $\mu + \lambda_k$-weight space. Similarly, operating by $\partial_k$ takes the $\mu$-weight space to the $\mu-\lambda_k$-weight space. It follows that $m'$ is either zero or is a weight vector of weight $\lambda'$ as required. But $m'$ is necessarily non-zero by the assumption that no weight vector is annihilated by either $x_k$ or $\partial_k$ for any $k \in \{1,2,\ldots , n\}$.
\end{proof}

Assume for a moment that $n = 1$. The delta module supported at $0$ is the irreducible $\dd$-module $\dd / \dd x_1 = \C[\partial_1] \delta_1$; here $\delta_1$ is the image of the unit in the quotient. The delta module can be endowed with a $T$-equivariant structure. This is uniquely defined by stipulating that $\delta_1$ is a weight vector of weight $\lambda_1$.  

For a finite subset $S\subseteq \Z^r$, we denote by $\mc P(S) = \mc P(\bigoplus_{\lambda \in S} \C_\lambda)$.

\begin{prop}
	There is a finite subset $S\subseteq \Z^r$ such that $\bigoplus_{\lambda \in S} \C_\lambda$ is a projective generator of $\Coh(\dd_X, T)$.
\end{prop}
\begin{proof}
We prove the proposition by induction on $r+n$, the sum of the rank of the torus and the dimension of $X$. Note that the statement is evident when $r=0$ or when $n=0$.

First suppose that $\lambda_1 \, \ldots, \lambda_n$ do not span $\Q^r$, or equivalently, the map $\mf t \to \C^n$ induced by the action is not injective. Then there is a non-trivial subtorus $Z$ of $T$ which acts trivially on $X$. In this case, the category $\Coh(\dd_X,T)$ is equivalent to $\Coh(\dd_X,T/Z)$. Thus the result follows from the inductive hypothesis in this case.

Now assume that $\lambda_1 \, \ldots, \lambda_n$ do span $\Q^r$, or equivalently, that $\Z^r/\langle \lambda_1,\ldots, \lambda_n\rangle$ is finite. Let $S'\subseteq \Z^r$ denote a set of representatives for each coset. Let $X_i = \{x_i = 0\} \subseteq X$. The action of $T$ restricts to each hyperplane $X_i$, and by the inductive hypothesis there is a finite subset $S_i \subseteq \Z^r$ such that $\bigoplus_{\lambda \in S_i} \C_\lambda$ is a projective generator for $\Coh(\dd_{X_i},T)$.

Now let $S$ be the subset of $\Z^r$ given by the union of:
\begin{itemize}
	\item $S_i$ for $i =1, \ldots, n$,
	\item $S_i -\lambda_i$ for $i=1, \ldots, n$; and
	\item $S'$
\end{itemize}
We will show that $S$ is the desired finite set of weights such that $\bigoplus_{\lambda \in S} \C_\lambda$ is a projective generator of $\Coh(\dd_X,T)$.

To show this, it suffices to show that for each $T$-equivariant $D$-module $\mc M$ on $X$, there is weight $\mu \in S$ such that $\mc M$ has a non-zero weight space of weight $\mu$. There are three possibilities: 
\begin{enumerate}
	\item There exists a weight vector $m \in \mc M$ such that $x_k^Nm=0$ for some $k\in \{1,2,\ldots, n\}$ and $N>0$. Replacing $\mc M$ by $\dd_X m$ if necessary, we may assume that $\mc M$ is supported on the hyperplane $X_k$. Thus by Kashiwara's Lemma \cite[Theorem 1.6.1]{HTT}, 
	\[
	\mc M = i_{k\ast}( \mc M') \simeq \C[\partial_k]\delta_k \otimes \mc M'
	\]
	where $\mc M'$ is a $\dd_{X_k}$-module, and $i_k$ denotes the inclusion $X_k \hookrightarrow X$. The $\dd_{X_k}$-module $\mc M'$ is necessarily $T$-equivariant, and thus by the inductive hypothesis has a weight vector $m'$ with weight in $\mu' \in S_k$. It follows that there is a weight vector of weight $\mu' - \lambda_k$ in $\mc M\simeq \C[\partial_k]\delta_k \otimes \mc M'$ of the form $\delta_k \otimes m'$.
	\item There exists a weight vector $m \in \mc M$ such that $\partial_k^Nm=0$ for some $k\in \{1,2,\ldots, n\}$ and $N>0$. Replacing $\mc M$ by $\dd_X m$ as before, we may assume the Fourier transform of $\mc M$ is supported on $X_k$ and thus
	\[
	\mc M = p_{k}^\circ( \mc M') \simeq \C[x_k] \otimes \mc M'
	\]
	where $\mc M'$ is a $\dd_{X_k}$-module and $p_k:X\to X_k$ is the projection. As before, there is a weight vector $m'$ of weight $\mu' \in S_k$ in $\mc M'$ and thus there is a weight vector of weight $\mu'$ in $\mc M$ of the form $1 \otimes m'$.
	\item If neither of the first two possibilities holds it follows from Lemma \ref{lem:everywheresupportedTmod} that for every weight vector $m \in \mc M$ of weight $\lambda$, there is a weight vector $m'$ of weight $\lambda'$ for each $\lambda'$ in the coset $\lambda + A$. 
	\end{enumerate}
In each of these three cases, there is a weight vector in $\mc M$ with weight in the given finite subset $S$, as required.
\end{proof}

\subsection*{Example}\label{ex:toruseqDmod1}

We finish the section with a simple example that none the less demonstrates well the general situation. Let $G = \Cs$ acting on $\mathbb{A}^1$ with weight $a \in \Z \smallsetminus \{ 0 \}$. Using the Fourier transform, we may assume without loss of generality that $a > 0$. Since there are only two $G$-orbits in $\mathbb{A}^1$, all equivariant $\dd$-modules are holonomic. By applying the results of \cite{BdMonvel} one can show that the category $\Coh(\dd_{\mathbb{A}^1},G)$ is in fact of finite type i.e. has only finitely many indecomposable objects up to isomorphism. Firstly,  
$$
\Coh(\dd_{\mathbb{A}^1},G) = \bigoplus_{\eta \in \Z / a \Z} \Coh(\dd_{\mathbb{A}^1},G)_{\eta},
$$
where $\Coh(\dd_{\mathbb{A}^1},G)_{\eta}$ is the full subcategory of all $\dd$-modules $\ms{M}$ with weights in the coset $\eta$. For any integer $b \in \eta$, we define the cyclic $G$-equivariant $\dd$-module  
$$
\mc{P}(b) = \dd_X \cdot v_b, \quad \textrm{ with } \quad a x \partial \cdot v_b = b v_b \quad \textrm{ and } \quad t \cdot v_b = t^b v_b \quad \forall \, t \in \Cs.
$$
It is projective in $\Coh(\dd_{\mathbb{A}^1},G)$. For $\eta \neq 0$ and any integer $b \in \eta$, the projective module $\mc{P}(b)$ is irreducible and defines an equivalence $\Coh(\dd_{\mathbb{A}^1},G)_{\eta} \simeq \mathrm{Vect}_{\C}$. The category $\Coh(\dd_{\mathbb{A}^1},G)_{0}$ contains the four indecomposable modules 
$$
\C[x], \quad \C[\partial] \otimes \delta_{-a}, \quad \mc{P}(0), \quad \mc{P}(-a).
$$   
We have $\mc{P}(0) \simeq \mc{P}(an)$ and $\mc{P}(-a) \simeq \mc{P}(-an)$ for all $n > 0$. This implies that $\Coh(\dd_{\mathbb{A}^1},G)_{0}$ is equivalent to $\Lmod{R_0}$, where $R_0 = \C Q / I$ with $Q$ the quiver 
$$
\begin{tikzcd}
0 \ar[r,"\alpha"',bend right] & 1 \ar[l,"\beta"', bend right] 
\end{tikzcd} 
$$
and $I$ the ideal of the path algebra $\C Q$ generated by the admissible relations $\{ \alpha \beta, \beta \alpha \}$. The non-split short exact sequences 
$$
0 \rightarrow  \C[\partial] \otimes \delta_{-a} \rightarrow \mc{P}(0) \rightarrow \C[x] \rightarrow 0 
$$
$$
0 \rightarrow  \C[x] \rightarrow \mc{P}(-1) \rightarrow \C[\partial] \otimes \delta_{-a} \rightarrow 0 
$$
show that $R_0$ has infinite global dimension.  

In total, 
$$
\mc{P} = \mc{P}(-a) \oplus \mc{P}(0) \oplus \mc{P}(1) \oplus \cdots \oplus \mc{P}(a-1)
$$
is a projective generator of $\QCoh(\dd_{\mathbb{A}^1},G)$ and the latter category is equivalent to $\LMod{R}$, where $R = R_0 \oplus \C^{\oplus (a-1)}$. 

\begin{remark}
	Global dimension: The category $\QCoh(\dd_X)$ of quasi-coherent $\dd_X$-modules has finite global dimension; see \cite{BjorkBook}. Despite the fact that $\QCoh(\dd_X,G)$ has enough projectives the above example explicitly shows that $\QCoh(\dd_X,G)$ can have infinite global dimension.     
\end{remark}

\section{A recollement pattern}

We fix a reductive connected algebraic group $G$, $T \subset G$ a maximal torus and $P = \Hom(T,\Cs)$ the corresponding weight lattice. Let $P^+$ be the subset of dominant weights, which we identify with the set of irreducible finite dimensional representations of $G$. The following is presumably well-known, but we were unable find a suitable reference so we include a complete proof.

\begin{lem}\label{lem:QCohGrothendieck}
	The category $\QCoh(\dd_X,G)$ is a Grothendieck category. 
\end{lem}

\begin{proof}
	Since $X$ is affine, we identify $\QCoh(\dd_X,G)$ with the category of all $\dd(X)$-modules $M$ that are also rational $G$-modules, the multiplication map $\dd(X) \o M \rightarrow M$ is $G$-equivariant and the differential of the $G$-action on $M$ agrees with the action of $\nu_X(\mf{g})$.  
	
	In this case, it is clear that $\QCoh(\dd_X,G)$ admits arbitrary direct sums since both the category of $\dd(X)$-modules and the category of rational $G$-modules admit arbitrary direct sums, which in both cases commute with the forgetful functor to vector spaces. Moreover, if $(A_i)_{i \in I}$ is an increasing directed family of subobjects of $A \in \QCoh(\dd_X,G)$ and $B$ any $G$-equivariant submodule of $A$ then $(\sum_{i \in I} A_i) \cap B = \sum_{i \in I} (A_i \cap B)$ since this already holds in the category of $\dd(X)$-modules. Finally, the object 
	$$
	\mc{P}_X := \bigoplus_{V \in \check{G}} \mc{P}_X(V)
	$$
	is a generator in $\QCoh(\dd_X,G)$. 
\end{proof}

We deduce that $\QCoh(\dd_X,G)$ admits injective envelops. 

\begin{remark}
	Lemma \ref{lem:QCohGrothendieck} implies that $\QCoh(\dd_X,G)$ admits arbitrary products, or more generally colimits. However, the reader is cautioned that these products differ from the corresponding products in $\QCoh(\dd_X)$. 
\end{remark}

For any subset $U \subset P^+$, we define $\QCoh(\dd_X,G)_{U}$ to be the full subcategory of $\QCoh(\dd_X,G)$ consisting of all modules $\ms{M}$ with $\Gamma(X,\ms{M})_{\lambda} = 0$ for all $\lambda \in U$. Let $\QCoh(\dd_X,G)(U)$ be the quotient category $\QCoh(\dd_X,G) / \QCoh(\dd_X,G)_{U}$. 

\begin{prop}\label{prop:recollementofweightsXGaffine}
	The quotient functor $j^* \colon \QCoh(\dd_X,G) \rightarrow \QCoh(\dd_X,G)(U)$ admits both a left adjoint $j_*$ and a right adjoint $j_!$. Similarly, the inclusion $i_* \colon \QCoh(\dd_X,G)_U \rightarrow \QCoh(\dd_X,G)$ admits left and right adjoints $i^*, i^!$. This forms a full recollement pattern
	\begin{equation}\label{eq:weightsrecollmentDXsmooth}
	\begin{tikzcd}
	\QCoh(\dd_X,G)_U \ar[r,"i_*"] & \QCoh(\dd_X,G) \ar[r,"j^*"] \ar[l,bend left,"i^*"] \ar[l,bend right,"i^!"'] & \QCoh(\dd_X,G)(U) \ar[l,bend left,"j_*"] \ar[l,bend right,"j_!"'].
	\end{tikzcd}
	\end{equation}
\end{prop}

\begin{proof}
	Showing the existence of the right adjoint $j_!$ is (by definition) equivalent to showing that $\QCoh(\dd_X,G)_U$ is a localizing subcategory, and showing the existence of $j_*$ is equivalent to showing that $\QCoh(\dd_X,G)_U$ is colocalizing. The fact that $\QCoh(\dd_X,G)_U$ follows by \cite[Proposition 2.2]{GeigleLenzingQuotientcat} from the fact that $\QCoh(\dd_X,G)_U$ admits arbitrary direct sums. 
	
	To show that $\QCoh(\dd_X,G)_U$ is colocalizing, we show that the dual of \cite[Proposition 2.2 (ii)]{GeigleLenzingQuotientcat} holds. Namely, if $\ms{M} \in \QCoh(\dd_X,G)$ then we must show that there exists an exact sequence 
	$$
	\ms{A} \rightarrow \ms{M} \rightarrow \ms{S} \rightarrow 0
	$$
	with $S \in \QCoh(\dd_X,G)_U$ and $A$ belongs to 
	$$
	{}^{\perp} \QCoh(\dd_X,G)_U = \{ \ms{B} \in \QCoh(\dd_X,G) \, | \, \Hom(\ms{B},\ms{S}) = \Ext^1(\ms{B},\ms{S}) = 0, \, \forall \ms{S} \in \QCoh(\dd_X,G)_U \}.
	$$
	Let $M' = \bigoplus_{\lambda \in U} \Gamma(X,\ms{M})_{\lambda}$ be the $G$-submodule of all isotypic components lying in $U$ and let $M(U)$ be the $\dd(X)$-submodule of $\Gamma(X,\ms{M})$ generated by $M'$. It defines a $G$-equivariant submodule $\ms{M}(U)$ of $\ms{M}$. Then $\ms{S} := \ms{M} / \ms{M}(U)$ is the largest quotient of $\ms{M}$ belonging to $\QCoh(\dd_X,G)_U$. Moreover, if we define 
	$$
	\mc{P}_{X}(U) = \bigoplus_{\lambda \in U}\mc{P}_{X}(V_\lambda)
	$$
	then $\mc{P}_{X}(U) \in {}^{\perp} \QCoh(\dd_X,G)_U$ and there exists a set $I$ such that 
	$$
	\mc{P}_X(U)^{(I)} \rightarrow \ms{M} \rightarrow \ms{S} \rightarrow 0
	$$
	is exact because $\mc{P}_X(U)^{(I)}$, the direct sum of $|I|$ copies of $\mc{P}_X(U)$, surjects onto $\ms{M}(U)$.  
\end{proof}

\begin{remark}
	If $A$ denotes the subgroup of $P$ generated by all weights of $T$ in $\C[X]$ then 
	$$
	\QCoh(\dd_X,G) = \bigoplus_{\eta \in P/A} \QCoh(\dd_X,G)(\eta)
	$$
	where $\QCoh(\dd_X,G)(\eta)$ is the full subcategory of all modules $\ms{M}$ whose $T$-weights all belong to the coset $\eta$.  
\end{remark}

Choosing a compact projective generator $\mc{Q}$ of $\QCoh(\dd_X,G)$, one can explicitly realize the functors appearing in Proposition \ref{prop:recollementofweightsXGaffine}. Namely, if $R := \End_{\QCoh(\dd_X,G)}(\mc{Q})^{\mathrm{op}}$, then there exists an idempotent $e \in R$ such that the equivalence $\QCoh(\dd_X,G) \stackrel{\sim}{\longrightarrow} \LMod{R}$ given by $\Hom_{\QCoh(\dd_X,G)}(\mc{Q}, - )$ induces equivalences 
$$
\QCoh(\dd_X,G)(U) \simeq \LMod{e R e}, \quad \QCoh(\dd_X,G)_{U} \simeq \LMod{(R / R e R)}.
$$
Under these equivalences the diagram \eqref{eq:weightsrecollmentDXsmooth} becomes  
$$
\begin{tikzcd}
\LMod{(R / R e R)} \ar[rr,hook] & & \LMod{R} \ar[rr,"e \cdot - "] \ar[ll,bend right,"{\Hom_R(R/ReR, - )}"'] \ar[ll,bend left,"{(R / R eR) \o_{R} - }"] & & \ar[ll,bend left,"{R e \otimes_{eRe} - }"] \ar[ll,bend right,"{\Hom_{eRe}(eR, - )}"'] \LMod{e R e} .
\end{tikzcd}
$$ 
This clearly illustrates that one cannot expect a similar recollement pattern for $\Coh(\dd_X,G)$ in general.

\section{Compact generators for the category of $D$-modules on a finite type Artin stack}\label{sec:proof-of-theorem-refthmdg-compact-generator}


In this section we will give a proof of Theorem \ref{thm:dg compact generator}.
This proof is very similar to that of Theorem \ref{thm:derived}, but expressed in more abstract terms. First we record the following:
\begin{lem}[\cite{drinfeld_finiteness_2013}, Lemma 10.3.9]\label{lem:drinfeld-gaitsgory}
	Given a finite type Artin stack $\cY$ with affine diagonal, there is a diagram:
	\[
\cY	\supset \mathring{\cY} \leftarrow \cZ \to X \times BG
\]
where
\begin{itemize}
	\item $\mathring{\cY} \supset \cY$ is a non-empty open substack,
	\item $\cZ \to \cY$ is a finite \'etale covering,
	\item $X$ is a scheme and $G$ a reductive group,
	\item The morphism $\cZ \to X\times BG$ is a unipotent gerbe.
\end{itemize}
\end{lem}

\begin{prop}\label{prop:strat of stacks}
Suppose $\cX$ is a finite type Artin stack with affine diagonal. Then there is a sequence of closed substacks 
	\[
	\cY_0 \subseteq \ldots \subseteq \cY_n = \cX
	\]
such that the successive complements $\cZ_i = \cY_i - \cY_{i-1}$ have the property that $\bD(\cX_i)$ admits a compact generator.
\end{prop}
\begin{proof}
We define the substacks $\cY_i$ and $\cX_i$ inductively as follows. We set $\cY_n = \cX$. Then define $\cX_i$ to be the open substack $\mathring{\cY_i}$ guaranteed by Lemma \ref{lem:drinfeld-gaitsgory} and $\cY_{i-1} = \cY_i - \cX_i$ to be the complementary (reduced) closed substack.

Now set $\cY = \cY_i$ and use the notation in Lemma \ref{lem:drinfeld-gaitsgory}. By Lemma \ref{lem:special case} below, we have that $\bD(X \times BG) \simeq \bD(X) \otimes \bD(BG)$ has a compact generator. As the morphism $\widetilde{\cU} \to X\times BG$ is a unipotent gerbe, it induces an equivalence $\bD(\widetilde{\cU}) \simeq \bD(X\times BG)$ (see \cite{drinfeld_finiteness_2013} Lemma 10.3.6). Finally, as $f$ is a finite \'etale cover, the functor $f^! \simeq f^\ast: \bD(\cU) \to \bD(\widetilde{\cU})$ is conservative, and thus $f_\ast$ takes the compact generator of $\bD(\widetilde{U})$ to the compact generator of $\bD(\cU)$ as required. 
\end{proof}

\begin{lem}\label{lem:special case}
	The DG category $\bD(\cX)$ has a compact generator in the following cases:
	\begin{enumerate}
		\item $\cX = X$ is a finite type scheme,
		\item $\cX = BG$ is the classifying stack of an affine algebraic group.
	\end{enumerate}
\end{lem}
\begin{proof}
	\begin{enumerate}
		\item First we recall that the derived category $\QC(X)$ of a quasi-coherent sheaves has a compact generator $\cF$ \cite{bondal_generators_2003} Theorem 3.1.1. By \cite{gaitsgory_crystals_2011} Proposition 3.4.11, the forgetful functor $\res:\bD(X) \to \QC(X)$ is conservative and colimit preserving. It follows that the left adjoint $\ind$ takes the compact generator $\cF$ to a compact generator of $\bD(X)$. 
		\item It is well-known that $\pi_\ast(\dd_{pt})$ is a compact generator, where $\pi:pt \to BG$ is the tautological $G$-bundle. Indeed, we have $R\End_{\bD(pt)}(\pi_\ast (\dd_{pt})) \simeq C_\ast(G)$ and thus $\bD(BG) \simeq C_\ast(G)-mod$ is identified with the category of dg modules for the dg algebra of chains on $G$ (see \cite{drinfeld_finiteness_2013} Section 7.2.2).
	\end{enumerate}
\end{proof}

The inductive step is then taken care of by the following:
\begin{lem}\label{lem:open closed}
	Let $\cZ$ be an closed substack of $\cX$ and $\cU$ the open complement. If the DG categories $\bD(\cU)$ and $\bD(\cZ)$ have a compact generator, then so does $\bD(\cX)$.
\end{lem}
\begin{proof}
	Let $j: \cU \hookrightarrow \cX$ and $i: \cZ \hookrightarrow \cX$ denote the inclusion morphisms. By assumption, $\bD(\cU)$ has a compact generator $\ms{M}$, and $\bD(\cZ)$ has a compact generator $\ms{N}$. By \cite{drinfeld_finiteness_2013} Theorem 0.2.2, $\bD(\cX)$ has a collection of compact generators $\{L_i\}_{i \in I}$ indexed by a set $I$. It follows immediately from adjunction properties that $\{j^\ast(L_i)\}_{i \in I}$ form a set of compact generators of $\bD(\cU)$. Moreover, as the object $\ms{M}$ is compact, it must be contained in the Karoubian-triangulated envelope of some finite subset of generators $\{j^\ast L_{i_1}, \ldots, j^\ast L_{i_n}\}$, which in turn must generate $\bD(\cU)$. 
	
	We will show that the object 
	\[
	K = i_\ast \ms{N} \oplus L_{i_1} \oplus \ldots \oplus L_{i_n}
	\]
	is a compact generator of $\bD(\cX)$. Indeed, given any object $S \in \bD(\cX)$, there is a distinguished triangle 
	\[
	i_\ast i^! S \to S \to j_\ast j^\ast S \xrightarrow{+1}
	\]
	We need to show that $R\Hom(K,S) \neq 0$. First suppose that $i^!S \neq 0$. Then 
	\[
	R\Hom(i_\ast \ms{N}, S) = R\Hom(\ms{N},i^!S) \neq 0
	\]
	as $\ms{N}$ is a generator of $\bD(\cZ)$. Now suppose that $i^! S=0$. Then $S \simeq j_\ast j^\ast S$, so we have
	\[
	R\Hom(L_{i_1} \oplus \ldots \oplus L_{i_n},S) \simeq R\Hom(j^\ast L_{i_1} \oplus \ldots \oplus j^\ast L_{i_n}, j^\ast S) \neq 0
	\]
	as $L_{i_1} \oplus \ldots L_{i_n}$ is a generator of $\bD(\cU)$. In either case we have $R\Hom(K,S) \neq 0$ as required. 
\end{proof}

Theorem \ref{thm:dg compact generator} now follows directly from Proposition \ref{prop:strat of stacks} together with Lemma \ref{lem:open closed}.

\def\cprime{$'$} \def\cprime{$'$} \def\cprime{$'$} \def\cprime{$'$}
\def\cprime{$'$} \def\cprime{$'$} \def\cprime{$'$} \def\cprime{$'$}
\def\cprime{$'$} \def\cprime{$'$} \def\cprime{$'$} \def\cprime{$'$}
\def\cprime{$'$} \def\cprime{$'$}

\end{document}